\documentclass[english]{amsart}
\usepackage{babel}
\usepackage{dsfont}
\usepackage{amsmath,amsthm,amssymb,babel,amsfonts,graphics}
\usepackage{caption}
\usepackage{subcaption}
\usepackage{tikz}
\usepackage{esint}
\usepackage[pdftex,colorlinks]{hyperref}
\usepackage{hyperref}
\usepackage{pgfplots}
\pgfplotsset{compat=1.16}

\title[Consistency of tug-of-war type operators]{Consistency of tug-of-war type operators on random data clouds}

\author[Han]{Jeongmin Han}
\address{Jeogmin Han: Department of Mathematics, Soongsil University, 
06978 Seoul, Republic of Korea}
\email{jeongmin.han@ssu.ac.kr}

\author[Liu]{Huajie Liu}
\address{Huajie Liu: Department of Mathematics and Statistics, University of Jyv\"{a}skyl\"{a}, P.O. Box 35, FI-40014 Jyv\"{a}skyl\"{a}, Finland \\
and Department of Mathematics, Nankai University, 300071 Tianjin, People's Republic of China}
\email{1120220031@mail.nankai.edu.cn}

\date{July 24, 2025}


\begin{document}

\maketitle

\theoremstyle{definition}
	\newtheorem{theorem}{Theorem}[section]
	\newtheorem{lemma}[theorem]{Lemma}
	\newtheorem{definition}[theorem]{Definition}
	\newtheorem{proposition}[theorem]{Proposition}
    \newtheorem{corollary}[theorem]{Corollary}
\theoremstyle{remark*}
    \newtheorem*{remark}{Remark}
\begin{abstract}
    In this paper, we study a tug-of-war type operator on geometric graphs and its associated Dirichlet problem on a random data cloud.
    Specifically, we analyze the convergence of the value functions as the number of data points increases and the step size of the game shrinks. This analysis reveals the connection between our tug-of-war type operator and the corresponding model problem.
    A key ingredient in establishing this result is the consistency of the operator. 
\end{abstract}

\tableofcontents

\section{Introduction}
In this paper, we investigate a class of discrete dynamic programming equations on random geometric graphs, denoted by $
\{ \mathcal{X}_n \}$, in $\mathbb{R}^N$:
\begin{align} \label{gradpp} \begin{split}
 u(x)=&\frac{1}{\epsilon^2}\biggl( \frac{p-2}{2(N+p)}\left\{ 
{\inf_{ Z_i\in \{B_{\epsilon}(x)\cap \mathcal{X}_n\} }  u(Z_i) +\sup_{ Z_j\in \{B_{\epsilon}(x)\cap \mathcal{X}_n\} }u(Z_j)} \right\}\\ & +
\frac{N+2}{N+p} \sum_{ Z_i\in \{B_{\epsilon}(x)\cap \mathcal{X}_n\} } \frac{u(Z_i)}{ |{ \{B_{\epsilon}(x)\cap \mathcal{X}_n\} }|} \biggr)-f. \end{split}
\end{align}
where $f\in L^{\infty}(\mathcal{X}_n)$. Such graphs arise in graph-based machine learning, constructed from random data points sampled from a distribution supported on a bounded domain in $\mathbb{R}^N$ (cf. \cite{Cal,ABP24}). Our focus is to analyze the behavior of the solutions to \eqref{gradpp} as the graph becomes dense.

In PDE theory, the mean value property leads to numerous important and insightful consequences for the Laplace equation. 
Moreover, it provides valuable intuition about the local behavior of harmonic functions. 
This naturally motivates the extension of the mean value property to broader classes of PDEs, offering new characterizations of their solutions. 
In particular, this perspective reveals a deep connection between stochastic game theory and PDEs, enriching our understanding of solutions through probabilistic and game-theoretic interpretations.

A notable example in this context is tug-of-war, which is a stochastic game first introduced by Peres, Schramm, Sheffield, and Wilson in the mid-2000s.
This game provides an interpretation of the infinity Laplacian, as shown in \cite{Peres2008}. 
Over the past two decades, there has been active research on this class of stochastic games, leading to a rapid development of the corresponding theoretical framework.
In many cases, the study of tug-of-war games begins with an investigation of the corresponding mean value characterization.
Mean value characterizations for the $p$-Laplace type operators have been developed, such as in \cite{MR2566554,MR2684311,MR2875296}, describing the local behavior of the solutions to these model equations. 
On the other hand, one can find interesting and important results for the value functions of tug-of-war games. 
For instance, in \cite{MR2451291,MR2887928}, tug-of-war with noise is examined via a dynamic programming principle (DPP), yielding a probabilistic interpretation of $p$-harmonic functions.
Similar approaches can be found in \cite{MR2887928,MR3494400,MR3623556,MR4299842}.
For a broader overview of tug-of-war games, we refer the reader to \cite{Par23,Lewicka2020,Blanc2021}.

On the other hand, DPPs are also pivotal in the study of PDEs on graphs, where discrete versions of these principles are employed. These discrete principles have numerous applications in machine learning, including large-scale regression, semi-supervised learning, and classification problems (cf. \cite{machlear1,Bert2009}).
There are several noteworthy studies on discrete DPPs, such as \cite{Cal,Calder22,calder2022hamilton}.
Additionally, Arroyo, Blanc and Parviainen investigated Krylov-Safonov type regularity for Pucci-type extremal operators on random clouds in \cite{ABP24}, which serves as a foundation for our discussion on the consistency of the tug-of-war type operator on random data clouds. 

Our work builds on the contributions of Calder, as well as those of Arroyo, Blanc, and Parviainen. In \cite{Cal}, Calder investigated the game-theoretic $p-$Laplacian in the context of semi-supervised learning on graphs with few label--an important problem in data science and machine learning. It is well known that for general domains, the Dirichlet problem for the $p-$Laplace equation may admit no solution when $p\leq N$, due to insufficient boundary capacity. Consequently, \cite{Cal} restricted attention to the case $p>N$. To explore whether the conclusions in \cite{Cal} extend to domains with sufficiently regular boundaries, we introduce additional geometric assumptions on the domain $\Omega$. By leveraging the regularity theory developed in \cite{ABP24}, we are able to extend the range to $p\geq 2$. Moreover, we also generalize the conclusions in \cite{Cal} and \cite{ABP24} to the non-harmonic case, where $f\neq 0$.

The main result of this paper, stated in Section \ref{nm}, demonstrates that as the random geometric graphs asymptotically fill the domain, the discrete tug-of-war operator, defined by the above dynamic programming equation, is consistent with the weighted $p$-Laplace operator given by 
\[
\frac{\phi^{-2}}{|\nabla u|^{p-2}}\mathrm{div}(\phi^2|\nabla u|^{p-2}\nabla u)
\]
for a given weight function $\phi$. 
Furthermore, we show that solutions of \eqref{gradpp} converge to the viscosity solution of the corresponding weighted $p$-Laplace equation. 
Our approach involves decomposing the discrete tug-of-war operator into two parts. One part is shown to be consistent with the weighted Laplacian via the Bernstein inequality, while the other part is shown to converge to the infinity Laplacian based on \cite{Cal}. To prove the convergence of the solutions sequence, we utilize an Ascoli-Arzel\'{a} type criterion, based on the regularity results in \cite{ABP24}, and then directly verify that the limit is a viscosity solution.

\subsection{Notation and main result}\label{nm}
In this part, we describe our setup and main result, Theorem \ref{mainthm}. Throughout this paper, we assume that  $\Omega\subset \mathbb{R}^N$ is a bounded domain satisfying the uniform exterior ball condition as follows:
\begin{definition}\label{ueb}
   A bounded domain $\Omega\subset \mathbb{R}^N$ is said to satisfy the uniform exterior ball condition, that is,
for any $y \in \partial \Omega$, there exists $ B_{\delta}(z) \subset \mathbb{R}^{n} \backslash \Omega $ with $\delta > 0 $ such that $y \in  \partial B_{\delta}(z)$. 
\end{definition}
We define the inner $\epsilon$-strip of $\partial \Omega$ as
    \[
    \Gamma_{\epsilon}:=\{ x\in \Omega:\text{dist}(x,\partial \Omega)\leq \epsilon \}.
    \]
Let $\mathcal{X}_n=\{ 
Z_1,\dots,Z_n \}\subset \Omega$ be a random data cloud where each $Z_i$ is an independent and identically distributed random variable with Lipschitz density $\phi:\Omega\rightarrow[\phi_0,\phi_1],0<\phi_0\leq\phi_1<\infty$.
Set $B_{\epsilon}^{\mathcal{X}}(x)=\mathcal{X}_n\cap B_{\epsilon}(x)$. Let $p\geq 2$ and set $\alpha=\frac{p-2}{N+p}\in[0,1)$ and $ \beta=\frac{N+2}{N+p}\in (0,1]$. Then we can check that $\alpha+\beta=1$.
The tug-of-war type operator $\mathcal{L}_{\mathcal{X}_n,\epsilon}$ is given by
\begin{align*}
 \mathcal{L}_{\mathcal{X}_{n},\epsilon} u(x)&=\frac{\alpha}{2\epsilon^2}\left\{ 
\inf_{ Z_i\in B_{\epsilon}^{\mathcal{X}_{n} }(x) }  u(Z_i) +\sup_{ Z_j\in B_{\epsilon}^{\mathcal{X}_{n} }(x) }u(Z_j)-2u(x) \right\}\\ & \qquad\qquad+
\frac{\beta}{ card (B_{\epsilon}^{\mathcal{X}_{n}} (x))} \sum_{Z_i\in B_{\epsilon}^{\mathcal{X}_{n} }(x)} \frac{u(Z_i)-u(x)}{\epsilon^2}.
\end{align*}
We note that the relation between the operator $\mathcal{L}_{\mathcal{X}_{n},\epsilon}  $ and the tug-of-war game can be heuristically inferred by considering its Taylor expansion.

Now we state our main theorem here.
\begin{theorem} \label{mainthm}
    Let $\{ n_k \}_{k\ge 1}\subset\mathbb{N}$ and $\{\epsilon_k \}_{k\ge1}\subset (0,\epsilon_0)$ be a sequence of positive numbers such that 
    \begin{align}\label{con4}
        n_k\rightarrow {\infty}\,,\,\epsilon_k\rightarrow {0}\,\, \textrm{and}\,\, \liminf_{k} (2n_k \textrm{exp}\{ -c_0 n_k\epsilon_k^{3N+4+(N+2)a} \})<1
    \end{align}
    for some constant $a>0$ which will be introduced later.
For each $k\in \mathbb{N}$, let $\mathcal{X}_{n_k}^{(k)}\subset \Omega$ be a random data cloud in $\Omega$ and denote $\mathcal{O}_{\epsilon_k}=\mathcal{X}_{n_k}^{(k)}\cap \Gamma_{\epsilon_k}$.
Let $u_k\in L^{\infty}(\mathcal{X}_{n_k}^{(k)})$ satisfy 
    \begin{align}\label{dpp}
\left\{ \begin{array}{ll}
   \mathcal{L}_{\mathcal{X}_{n_k}^{(k)},\epsilon_k} u_k=f_k \,\,&\textrm{in} \,\, \mathcal{X}^{(k)}_{n_k}\setminus \mathcal{O}_{\epsilon_k},\\
    u_k=g\quad \quad \quad &\textrm{in}\,\,\mathcal{O}_{\epsilon_k}, \\
\end{array} \right.
\end{align}
where $\mathcal{L}_{\mathcal{X}_{n_k}^{(k)},\epsilon_k}$ stands for the tug-of-war operator on random graph $\mathcal{X}_{n_k}^{(k)}$, $g\in C(\Gamma_{\epsilon_0})$ and the function sequence $\{ f_k\}\subset L^{\infty}(\mathcal{X}_{n_k}^{(k)})$ satisfies the following assumptions:
\par $(1)$ there exists $C_f>0$ such that $|f_k\circ T_{\epsilon_k}(x)|<C_f$ for every $k$ and $x\in\overline{\Omega}$,
\par $(2)$ given $\eta>0$, there exists $r_0$ and $k_0$ such that 
\[
|f_k\circ T_{\epsilon_k}(x)-f_k\circ T_{\epsilon_k}(y)|<\eta
\]
for every $k\geq k_0$ and $x,y\in \overline{\Omega}$ with $|x-y|<r_0$.
\\Then there exists a $f\in C(\overline{\Omega})$ such that the limit continuous function $u$ is the viscosity solution to 
    \begin{align}\label{modeleq}
\left\{ \begin{array}{ll}
  \frac{N+p}{N+2}\frac{\phi^{-2}}{|\nabla u|^{p-2}}\mathrm{div}(\phi^2|\nabla u|^{p-2}\nabla u)=f \,\,&\textrm{in} \,\,\Omega,\\
    u=g\quad \quad \quad \quad \qquad \quad&\text{in}\,\,\partial \Omega, \\
\end{array} \right.
\end{align}
with probability $1$.
\end{theorem}
The proof of our main result proceeds as follows. In Section 2, we are concerned with the consistency of the tug-of-war operator $ \mathcal{L}_{\mathcal{X}_{n_k}^{(k)},\epsilon_k}$  with the weighted $p$-Laplacian
\begin{align*}
    \frac{\phi^{-2}}{|\nabla \varphi|^{p-2}}\mathrm{div}(\phi^2|\nabla \varphi|^{p-2}\nabla \varphi)= \frac{1}{\phi^2}\mathrm{div}(\phi^2\nabla \varphi)+(p-2)\Delta_{\infty}\varphi.
\end{align*}
More precisely, we focus on establishing the consistency of the operators $\mathcal{L}_2^k$ and $\mathcal{L}^k_{\infty}$ (see the definition \eqref{l2} and \eqref{lin})
since the operator $ \mathcal{L}_{\mathcal{X}_{n_k}^{(k)},\epsilon_k} $ can be represented by a linear combination of those two operators. 
The Taylor expansion of those operators provides the key idea of the proof of consistency.
We also investigate the convergence of the solution to \eqref{dpp} as $\epsilon_k \to 0$ in Section 3. To do this, we define the extension of our solutions (see Definition \ref{tmap}), and then use an Ascoli-Arzel\'{a} type criterion (Lemma \ref{conv}). With the consistency and convergence, we finally obtain that the limit $u$ solves the model problem \eqref{modeleq} in a viscosity sense.

\section{Consistency}
There have been a number of preceding results for the consistency between operators on graphs and general differential operators in $\mathbb{R}^N$. For example, we can refer the reader to the paper \cite{ABP24} for a result about the asymptotic expansion of the Pucci-type maximal operator on graphs in $\mathbb{R}^N$. 
In \cite{Cal}, one can also find a consistency result for a class of operators on graphs, which is called the game-theoretic graph $p$-Laplacian. 
We employ and improve the method used in \cite{Cal} to obtain our consistency result.    
\par The key ingredient of this section is the following theorem, which gives us the consistency of $  \mathcal{L}_{\mathcal{X}_{n_k}^{(k)},\epsilon_k} $.
\begin{lemma}\label{six}
    For all $x_0\in \Omega$ and $\varphi\in C^{\infty}(\mathbb{R}^N)$ with $\nabla \varphi (x_0)\neq 0$,
    \begin{align*}
        \lim_{\substack{k\rightarrow\infty\\x\rightarrow x_0}  }
        \mathcal{L}_{\mathcal{X}_{n_k}^{(k)},\epsilon_k} \varphi(x)=\frac{N+p}{N+2}\phi^{-2}|\nabla\varphi|^{2-p}\mathrm{div}(\phi^2|\nabla \varphi|^{p-2}\nabla \varphi)|_{x=x_0}
    \end{align*}
    with probability $1$.
\end{lemma}
For convenience, we denote $\mathcal{L}^k_2$, $\mathcal{L}^k_{\infty}$ by 
\begin{align} \label{l2}
    \mathcal{L}^k_2 \varphi(x)&=\frac{1}{ card \big(B_{\epsilon_k}^{\mathcal{X}_{n_k}^{(k)}} (x)\big)} \sum_{Z_i\in B_{\epsilon_k}^{\mathcal{X}_{n_k}^{(k)}} (x)} \frac{\varphi(Z_i)-\varphi(x)}{\epsilon_k^2} \end{align} and
 \begin{align}   \label{lin}
    \mathcal{L}^k_{\infty}\varphi(x)&=\frac{1}{2\epsilon_k^2}\left\{ 
\inf_{ Z_i\in B_{\epsilon_k}^{\mathcal{X}_{n_k}^{(k)}} (x) }  \varphi(Z_i)+\sup_{ Z_j\in B_{\epsilon_k}^{\mathcal{X}_{n_k}^{(k)}} (x) }\varphi(Z_j)-2\varphi(x) \right\}
\end{align}
for $\varphi\in C^{\infty}(\Omega)$, respectively. Then we have 
\begin{align}\label{dep}
      \mathcal{L}_{\mathcal{X}_{n_k}^{(k)},\epsilon_k} \varphi(x)=\alpha\mathcal{L}^k_2 \varphi(x)+  \beta\mathcal{L}^k_{\infty}\varphi(x),
\end{align}
where $\alpha=\frac{p-2}{N+p}$ and $\beta=\frac{N+2}{N+p} $.
Since $\mathcal{L}_{\mathcal{X}_{n_k}^{(k)},\epsilon_k} $  can be represented as a convex combination of $\mathcal{L}^k_2\varphi(x)$ and $\mathcal{L}^k_{\infty}\varphi(x)$, it suffices to consider consistency of $\mathcal{L}^k_2$ and $\mathcal{L}^k_{\infty}$.

\subsection{Consistency of $\mathcal{L}^k_2$}
In this subsection, we investigate the consistency of $\mathcal{L}^k_2$.
Mean value characterizations based on Taylor expansions are powerful tools for analyzing stochastic games. In the Euclidean setting, a number of such characterizations have been developed (see \cite{MR2566554,MR2684311,MR2875296}, for example). These approaches are still useful when considering stochastic games on graphs.

To establish our consistency result, we begin by stating several lemmas that will be used throughout this subsection. The first is a Bernstein-type inequality for random variables on geometric graphs, which can be found in \cite[Lemma 1]{Cal}.
\begin{lemma}\label{bern}
    Let $Y_1,\dots,Y_n$ be a sequence of i.i.d random variables on $\mathbb{R}^N$ with Lebesgue density $f:\mathbb{R}^N\rightarrow\mathbb{R}$, let $\psi:\mathbb{R}^N\rightarrow\mathbb{R}$ be bounded and Borel measurable with compact support in a bounded domain $\Omega\subset \mathbb{R}^N$, and define 
    \[
    Y=\sum_{i=1}^n \psi(Y_i).
    \]
    Then for any $0\leq \lambda\leq 1$, we have
    \begin{align*}
        \mathbb{P}(|Y-\mathbb{E}(Y)|>\| f \|_{\infty} \| \psi \|_{\infty}n|\Omega|\lambda )\leq 2\exp\bigg\{ -\frac{1}{4}\| f \|_{\infty} n|\Omega|\lambda^2 \bigg\}
    \end{align*}
    with $\|\psi\|_{\infty}=\| \psi \|_{L^{\infty}(\Omega)}$, and $|\Omega|$ is the Lebesgue measure of $\Omega$.
\end{lemma} 
The following is a well-known, useful result to estimate probabilities.
\begin{lemma}\label{boca}(Borel-Cantelli)
    Let $(\Omega,\mathcal{F},\mathbb{P})$ be a probability space and $A_1,A_2,\dots \in \mathcal{F}$.
    Then the following are true:
    \par (1) If $\sum_{n=1}^{\infty} \mathbb{P}(A_n)<\infty$, then $\mathbb{P}(\limsup_{n\rightarrow\infty}A_n)=0$.
    \par (2) If $A_1,A_2,\dots$ are independent and $\sum_{n=1}^{\infty} \mathbb{P}(A_n)=\infty$, then $$\mathbb{P}(\limsup_{n\rightarrow\infty}A_n)=1.$$\\
    Consequently, if $A_1,A_2,\dots$ are independent, then
    \[
    \sum_{n=1}^{\infty} \mathbb{P}(A_n)<\infty\textrm{ if and only if } \mathbb{P}(\limsup_{n\rightarrow\infty}A_n)=0
    \]
    and
     \[
    \sum_{n=1}^{\infty} \mathbb{P}(A_n)=\infty\textrm{ if and only if } \mathbb{P}(\limsup_{n\rightarrow\infty}A_n)=1.
    \]
\end{lemma}
\par We first
Observe that
\[
\mathbb{E}(\mathcal{L}^k_2\varphi(x))=\frac{1}{\epsilon_k^2 
\int_{B_{\epsilon_k}(x)}\phi 
}\left({\int_{B_{\epsilon_k}(x)} (\varphi(y) }-\varphi(x))\phi(y)dy   \right).
\]
By Lemma \ref{bern} we have 
\begin{align*}
    \left|\mathcal{L}^k_2\varphi(x)-  \mathbb{E}(\mathcal{L}^k_2\varphi(x))\right|\geq C\delta n_k \epsilon_k^{N}
\end{align*}
occur with probability at most 2$\exp(-c\delta^2n_k\epsilon_k^N)$ provided $0<\delta\leq1$.
\par Set $\gamma=\|\varphi\|_{C^3(B_{\epsilon_k}(x))}$. Using Taylor expansion, we get
\begin{align}\label{taylor}
\begin{split}
\mathcal{L}^k_2\varphi(x)=&
\frac{\epsilon_k^N}{ \int_{B_{\epsilon_k}(x)}\phi }
\int_{B_{1}(0)} \left( \frac{1}{\epsilon_k}\nabla \varphi(x)\cdot z \right)\phi(x+\epsilon_k z)dz\\
&+\frac{\epsilon_k^N}{ \int_{B_{\epsilon_k}(x)}\phi }
\int_{B_{1}(0)} \left(\frac{1}{2}z\cdot \nabla^2 \varphi(x)z\right) \phi(x+\epsilon_k z)dz+o(\gamma) 
\end{split}
\end{align}
\noindent holds for all $\varphi\in C^{\infty}(\Omega)$ with probability at least $1-2\exp(-c\delta^2 n_k \epsilon_k^N )$.

Notice that
   \begin{align*}
       \frac{1}{\epsilon_k}\int_{B_1(0)}(\nabla \varphi(x)\cdot z)\phi(x+\epsilon_k z)dz&=\nabla \phi(x)\cdot \int_{B_1(0)}( \nabla\varphi(x)\cdot z )zdz+O(\epsilon_k\gamma)\\
       &=\nabla \phi(x)\cdot \sum_{i=1}^N\nabla_i\varphi(x)\int_{B_1(0)}z_izdz+O(\epsilon_k\gamma)\\
&=C_1 \nabla \phi(x)\cdot\nabla\varphi(x)+O(\epsilon_k\gamma),  
   \end{align*} 
   and 
\begin{align*}
    \frac{1}{2}\int_{B_{1}(0)} (z\cdot \nabla^2 \varphi(x)z)\phi(x+\epsilon_k z)dz&=\frac{1}{2}\phi(x) \sum_{i,j=1}^N \nabla^2_{ij}\varphi(x)\int_{B_1(0)} z_iz_j dz+O(\epsilon_k\gamma)\\
    &=\frac{1}{2}C_1 \phi(x)\text{tr}(\nabla^2\varphi)+O(\epsilon_k\gamma),
\end{align*}
where $C_1=\int_{B_1(0)} z_1^2dz$.
Combining this with \eqref{taylor}, we have 
\begin{align}
    \mathcal{L}^k_2\varphi(x)=\frac{C_1}{2}\frac{\epsilon_k^N}{ \int_{B_{\epsilon_k}(x)}\phi }\phi^{-1}\mathrm{div}(\phi^2\nabla \varphi)+O(\epsilon_k\gamma)+o(1)
\end{align}
holds with probability at least $1-2\exp(-c\delta^2 n_k \epsilon_k^N )$.
Combining \eqref{con4} with Lemma \ref{boca}, $\mathcal{L}_2^k \varphi$ converges to $\frac{C_1}{2\omega_N}\phi^{-2}\mathrm{div}(\phi^2\nabla \varphi)$ with probability $1$.

\subsection{Consistency of $\mathcal{L}^k_{\infty}$}
We focus on the following consistency result for $\mathcal{L}^k_{\infty}$ here.
\begin{lemma}\label{infty}
    \begin{align}\label{cons}
    \lim_{k\rightarrow\infty} \mathcal{L}^k_{\infty} \varphi (x)=\frac{1}{2}\Delta_{\infty}\varphi(x),
    \end{align}
    where
    \[
    \Delta_{\infty}\varphi =\frac{1}{|\nabla\varphi|^2} \sum_{i,j=1}^N \varphi_{x_ix_j}\varphi_{x_i}\varphi_{x_j}.
    \]
\end{lemma}
To this end, we first need to introduce a continuous function $\Phi :[0,2]\rightarrow \mathbb{R}$
satisfying:
\begin{align}\label{dec}
\Phi=1\text{ in }(0,1),\text{decreases to }0\text{ in }(1,2); 
    \Phi'(1)=\Phi''(1)=\Phi'(2)=\Phi''(2)=0.
\end{align}
For example,
\begin{align*}
\Phi(s)=\left\{ \begin{array}{lll}
1,& \textrm{ for } 0\le s\le 1,\\ 
(s-2)^3(-6s^2+9s-4),& \textrm{ for } 1<s<2,\\
    0,&\textrm{ for } s\geq 2. \\
\end{array} \right.
\end{align*}
Set $\omega(y,z)=\Phi(\frac{|y-z|}{\epsilon_k})$ for $ y,z\in B_{\epsilon
_k}(x)$.
Similarly to \cite{Cal}, we consider the following operators:
\begin{align*}
L_{\infty}^k\varphi(x)&=\left\{ 
\inf_{ Z_i\in {\mathcal{X}_{n_k}^{(k)}} } \omega(Z_i,x) (\varphi(Z_i)-\varphi(x))+\sup_{ Z_j\in {\mathcal{X}_{n_k}^{(k)}}  }\omega(Z_j,x)(\varphi(Z_j)-\varphi(x)) \right\},
\end{align*}
and
\begin{align*}
H^k\varphi(x)&=\left\{ 
\inf_{ Z_i\in \Omega } \omega(Z_i,x) (\varphi(Z_i)-\varphi(x))+\sup_{ Z_j\in \Omega }\omega(Z_j,x)(\varphi(Z_j)-\varphi(x)) \right\}.
\end{align*}
To derive the consistency, we recall the following lemma (see \cite[Lemma 1-2]{Cal}).
\begin{lemma}\label{coninf}
    Let $\varphi\in C^2(\mathbb{R}^N)$. Then
    \begin{align} \label{cons2}
        |L^k_{\infty}\varphi(x)-H^k\varphi(x)  |\leq C(\| \varphi \|_{C^1(B_x(2\epsilon_k))}+\epsilon_k\| \varphi \|_{C^2(B_x(2\epsilon_k))})r^2_k\epsilon_k^{-1},
    \end{align}
    where
    \begin{align*}
        r_n=\sup_{y\in \Omega} \text{dist}(y,\mathcal{X}^{(k)}_{n_k}).
    \end{align*}
\end{lemma}
In the proof of Lemma \ref{coninf}, it was assumed that $\Phi$ is $C^2$, which is satisfied by condition \eqref{dec}.
To this end, we also need to employ another result in \cite{Cal} (see Theorem 6 therein).
\begin{lemma}
    For any $x_0\in \Omega\setminus \mathcal{O}$ and $\varphi\in C^3(\mathbb{R}^N)$ with $\nabla\varphi(x_0)\neq 0$
    \begin{align} \label{cons3}
        \lim_{\substack{
        k\rightarrow\infty\\
          x\rightarrow x_0}
          }
          \frac{1}{\epsilon_k^2}H^k\varphi(x)=s_0^2\Phi(s_0)\Delta_{\infty}\varphi(x_0),
    \end{align}
    where $s_0$ is the unique maximum point of $s\Phi(s)$ in $[0,2]$.
\end{lemma}
It is not difficult to find a sequence of $\{\Phi_i\}_i$ that satisfies \eqref{dec} such that 
\[
\lim_{i\rightarrow\infty}\Phi_i=\mathbf{1}_{[0,1]},\lim_{i\rightarrow\infty} s_i^2 \Phi_i(s_i)=1,
\]
 where $\mathbf{1}_{[0,1]}$ denotes the indicator function of interval 
$(0,1)$ and $s_i$ is the maximum point of $s\Phi_i(s)$ in $[0,2]$. Then by \eqref{con2} and \eqref{con3}, we finally obtain \eqref{cons} in Lemma \ref{infty}.
\begin{remark}
Indeed,
\[
C_1=\int_{B_1(0)} z_1^2dz=\frac{1}{N}\int_{B_1(0)}z^2dz=\frac{1}{N}\int_{0}^1 N\omega_N r^{N+1}dr=\frac{\omega_N}{N+2}. 
\]
Combining the results of Sec. 2.1 and Sec. 2.2, we have $\mathcal{L}_{\mathcal{X}_{n_k}^{(k)},\epsilon_k} \varphi$ is consistent with 
\[
 \frac{(p-2)(N+2)}{N+p}\Delta_{\infty}\varphi+\frac{N+2}{N+p}{\phi^{-2}} \mathrm{div}(\phi^2\nabla \varphi)
\]
since $\alpha=\frac{p-2}{p+N},\beta=\frac{N+2}{N+p}$.
Then we finally get the Lemma \ref{six}: $\mathcal{L}_{\mathcal{X}_{n_k}^{(k)},\epsilon_k} \varphi$ is consistent with 
\[
\frac{N+p}{N+2}\frac{1}{\phi^2}|\nabla\varphi|^{2-p}\mathrm{div}(\phi^2|\nabla \varphi|^{p-2}\nabla \varphi)=\frac{N+p}{N+2}( \frac{1}{\phi^2}\mathrm{div}(\phi^2\nabla\varphi)+(p-2)\Delta_{\infty}\varphi).
\]
\end{remark}

\section{Proof of Theorem \ref{mainthm}}\label{pfthm}

In this section, we give the proof of our main result, Theorem \ref{mainthm}.
For convenience, we recall our main result here.
\begin{theorem}
    Let $\{ n_k \}_{k\ge 1}\subset\mathbb{N}$ and $\{\epsilon_k \}_{k\ge1}$ be a sequence of positive numbers such that 
    \begin{align}\label{con0}
        n_k\rightarrow {\infty}\,,\,\epsilon_k\rightarrow {0}\,\, \textrm{and}\,\, \liminf_{k} (2n_k \textrm{exp}\{ -c_0 n_k\epsilon_k^{3N+4+(N+2)a} \})<1
    \end{align}
    for some constant $a>0$ which will be introduced later.
For each $k\in \mathbb{N}$, let $\mathcal{X}_{n_k}^{(k)}\subset \Omega$ be a random data cloud in $\Omega$ and denote $\mathcal{O}_{\epsilon_k}=\mathcal{X}_{n_k}^{(k)}\cap \Gamma_{\epsilon_k}$.
Let $u_k\in L^{\infty}(\mathcal{X}_{n_k}^{(k)})$ satisfy 
    \begin{align}\label{dpf}
\left\{ \begin{array}{ll}
   \mathcal{L}_{\mathcal{X}_{n_k}^{(k)},\epsilon_k} u_k=f_k \,\,&\textrm{in} \,\, \mathcal{X}^{(k)}_{n_k}\setminus \mathcal{O}_{\epsilon_k},\\
    u_k=g\quad \quad \quad &\textrm{in}\,\,\mathcal{O}_{\epsilon_k}, \\
\end{array} \right.
\end{align}
where $\mathcal{L}_{\mathcal{X}_{n_k}^{(k)},\epsilon_k}$ stands for the tug-of-war operator on random graph $\mathcal{X}_{n_k}^{(k)}$, $g\in C(\Gamma_{\epsilon_0})$ and the function sequence $\{ f_k\in L^{\infty}(\mathcal{X}_{n_k}^{(k)})\}$ satisfies the following assumptions:
\par $(1)$ there exists $C_f>0$ such that $|f_k\circ T_{\epsilon_k}(x)|<C_f$ for every $k$ and $x\in\overline{\Omega}$,
\par $(2)$ given $\eta>0$, there exists $r_0$ and $k_0$ such that 
\[
|f_k\circ T_{\epsilon_k}(x)-f_k\circ T_{\epsilon_k}(y)|<\eta
\]
for every $k\geq k_0$ and $x,y\in \overline{\Omega}$ with $|x-y|<r_0$.
\\Then there exists a $f\in C(\overline{\Omega})$ such that the limit continuous function $u$ is the viscosity solution to 
    \begin{align}\label{meq}
\left\{ \begin{array}{ll}
  \frac{N+p}{N+2}\frac{\phi^{-2}}{|\nabla u|^{p-2}}\mathrm{div}(\phi^2|\nabla u|^{p-2}\nabla u)=f \,\,&\textrm{in} \,\,\Omega,\\
    u=g\quad \quad \quad \quad \qquad \quad&\text{in}\,\,\partial \Omega, \\
\end{array} \right.
\end{align}
with probability $1$.
\end{theorem}
Now we only need to prove the following results:
\\\textbf{(R.1)} The solutions sequence $\{u_k\}$ of \eqref{dpf}, at least for a subsequence, has a limit function $u$.
\\\textbf{(R.2)} The limit $u$ is a viscosity solution to \eqref{meq}.

\subsection{Convergence}
In this part, we are concerned with the convergence of solutions to \eqref{gradpp}. 
To do this, we give an extension of the solution sequence and employ an Ascoli-Arzel\'{a} criterion.

First, we need to present a segmentation of $\Omega$ associated with the data cloud $\mathcal{X}_n$.
We refer the reader to the following result, \cite[Lemma 2.1]{ABP24}.
Roughly speaking, this lemma enables us to describe the event of the points being evenly distributed with respect to $\mu$.
\begin{lemma}\label{two}
    Let $\phi$ be a Lipschitz continuous probability density in $\overline{\Omega}$ and $\epsilon>0$.
    For a random data cloud $\mathcal{X}_n=\{ 
Z_1,\dots,Z_n \}\subset \Omega$ with density $\phi$, there exists a probabilty measure $\mu_{\epsilon}$ with density $\phi_{\epsilon}\in L^{\infty}(\Omega)$ and a partition $\mathcal{U}_1,\dots,\mathcal{U}_n$ of $\Omega$ such that
\begin{align}\label{con1}
Z_i\in \mathcal{U}_i\subset B_{\epsilon^3}(Z_i)\quad \text{and}\quad \mu_{\epsilon}(\mathcal{U}_i)=\int_{\mathcal{U}_i} \phi_{\epsilon}(y)dy=\frac{1}{n}    
\end{align}
and
\begin{align}\label{con2}
    \| \phi_{\epsilon}-\phi \|_{\infty}\leq C_0 \epsilon^2
\end{align}
with a probability of at least
\begin{align}\label{con3}
    1-2n\exp\{ -c_0n\epsilon^{3N+4} \},
\end{align}
where $C_0>0$ and $c_0=c_0(\Omega)$ are universal.
\end{lemma}
It is possible to improve the \eqref{con1}, \eqref{con2} and \eqref{con3}.
According to \cite[Section 9]{ABP24}, for some $a>0$, there holds
\begin{align}\label{con11}
Z_i\in \mathcal{U}_i\subset B_{\epsilon^{3+a}}(Z_i)\quad \text{and}\quad \mu_{\epsilon}(\mathcal{U}_i)=\int_{\mathcal{U}_i} \phi_{\epsilon}(y)dy=\frac{1}{n}    
\end{align}
and
\begin{align}\label{con31}
    \| \phi_{\epsilon}-\phi \|_{\infty}\leq C_0 \epsilon^{2+a}
\end{align}
with a probability of at least
\begin{align}
    1-2n\exp\{ -c_0n\epsilon^{3N+4+(N+2)a} \}
\end{align}
 (see (9.2)-(9.4) therein).

 On the other hand, according to \cite[Definition 2.2]{ABP24}, with the partition $\mathcal{U}_1,\dots,\mathcal{U}_n$ of $\Omega$ in Lemma \ref{two}, it allows us to define a projection from $\Omega$ to $\mathcal{X}_n$. 
 \begin{definition}\label{tmap}
     Let $\mathcal{X}_n=\{Z_1,\dots,Z_n\}\subset \Omega$ and a partition $\mathcal{U}_1,\dots,\mathcal{U}_n$ of $\Omega$. We define the transport map $T_{\epsilon}:\Omega\rightarrow \mathcal{X}_n$
    as the projection map satisfying 
         $T_{\epsilon}(x)=Z_i$ if and only if $x\in \mathcal{U}_i$.
 \end{definition}

We introduce the following Pucci-type operators to deal with the convergence issue of our tug-of-war operator.
\begin{align*}
     \mathcal{L}^+_{\mathcal{X}_{n},\epsilon} u(x)&=\frac{\alpha}{2\epsilon^2}\left\{ 
\max_{Z_i\in B^{\mathcal{X}_n}_{\Lambda\epsilon}(x)}\max_{Z_j\in B^{\mathcal{X}_n}_{\tau\epsilon^2}(2x-Z_i)}[u(Z_i)+u(Z_j)]-2u(x) \right\}\\ & \qquad\qquad +
\frac{\beta}{ card (B_{\epsilon}^{\mathcal{X}_{n} }(x) )} \sum_{Z_i\in B_{\epsilon}^{\mathcal{X}_{n}} (x)} \frac{u(Z_i)-u(x)}{\epsilon^2}
\end{align*}
and
\begin{align*}
     \mathcal{L}^-_{\mathcal{X}_{n},\epsilon} u(x)&=\frac{\alpha}{2\epsilon^2}\left\{ 
\min_{Z_i\in B^{\mathcal{X}_n}_{\Lambda\epsilon}(x)}\min_{Z_j\in B^{\mathcal{X}_n}_{\tau\epsilon^2}(2x-Z_i)}[u(Z_i)+u(Z_j)]-2u(x) \right\}\\ & \qquad\qquad +
\frac{\beta}{ card (B_{\epsilon}^{\mathcal{X}_{n} }(x) )} \sum_{Z_i\in B_{\epsilon}^{\mathcal{X}_{n}} (x)} \frac{u(Z_i)-u(x)}{\epsilon^2}.
\end{align*}
Then it is not difficult to observe
\begin{align}\label{comparison}
\mathcal{L}^-_{\mathcal{X}_n,\epsilon}u\leq \mathcal{L}_{\mathcal{X}_n,\epsilon}u\leq \mathcal{L}^+_{\mathcal{X}_n,\epsilon} u.
\end{align}

 Next, we also present an Ascoli-Arzel\'{a} type lemma, which will be used to get the convergence of the solutions (see \cite[Lemma 4.2]{MPR} for the proof).
\begin{lemma}\label{aa}
    Let $\{ v_{\epsilon}:\overline{\Omega}\rightarrow\mathbb{R}\}_{\epsilon>0 }$ be a sequence of functions such that 
    \par (1) there exists $C_v>0$ so that $|v_{\epsilon}(x)|<C_v$ for every $\epsilon>0$ and every $x\in \overline{\Omega}$.
    \par (2) given $\eta>0$ there are constants $r_0$ and $\epsilon_0$ such that for every $\epsilon<\epsilon_0$ and $x$, $y\in \overline{\Omega}$ with $|x-y|<r_0$, it holds 
    \[
    |v_{\epsilon}(x)-v_{\epsilon}(y)|<\eta.
    \]
Then there a uniformly continuous function $v: \overline{\Omega}\rightarrow \mathbb{R}$ and a subsequence, still denoted by $\{ v_{\epsilon} \}$ such that 
    \begin{center}
        $v_{\epsilon}\rightarrow v$ uniformly in $\overline{\Omega}$ as $\epsilon\rightarrow 0$.
    \end{center}
\end{lemma}
To apply the above criterion, we need proper regularity results for the solutions. 
The following results provide uniform boundedness, H\"{o}lder regularity and boundary continuity of the solution.
The proof of those lemmas can be found in \cite[Lemma 9.3 - 9.5]{ABP24}.
\begin{lemma}\label{ub}
    Let $\mathcal{X}_n\subset \Omega$ be a random data cloud. Suppose that the events \eqref{con11} and \eqref{con31} hold and that $u\in L^{\infty}(\mathcal{X}_n)$ satisfies 
    \[
    \mathcal{L}^{+}_{\mathcal{X}_n,\epsilon}u\geq -\rho,\mathcal{L}^{-}_{\mathcal{X}_n,\epsilon}u\leq \rho
    \]
    in $\mathcal{X}_n\setminus \mathcal{O}_{\epsilon}$ for some $\rho>0$. Then 
    \begin{align*}
        \| u \|_{L^{\infty}(\mathcal{X}_n)}\leq \| u \|_{L^{\infty}(\mathcal{O}_{\epsilon})} +C_{\Omega}\rho
    \end{align*}
    for some constant $C_{\Omega}>0$ depending exclusively on $\Omega$.
\end{lemma}
\begin{lemma}\label{ac}
    Let $\Omega\subset \mathbb{R}^N$ be a bounded domain satisfying the uniform exterior ball condition and $\mathcal{X}_n\subset \Omega$ be a random data cloud. Suppose that the events \eqref{con11} and \eqref{con31} hold and that $u\in L^{\infty}(\mathcal{X}_n)$ satisfies 
     \begin{align*}
\left\{ \begin{array}{ll}
   \mathcal{L}^+_{\mathcal{X}_{n},\epsilon} u\geq -\rho,\mathcal{L}^-_{\mathcal{X}_{n},\epsilon} u\leq \rho \,\,&\textrm{in} \,\, \mathcal{X}_{n}\setminus \mathcal{O}_{\epsilon},\\
    u=g\quad \quad \quad &\textrm{in}\,\,\mathcal{O}_{\epsilon}, \\
\end{array} \right.
\end{align*}
for some $\rho>0$, where $g\in C(\Gamma_{\epsilon})$. Then for every $Z_i$, $Z_j\in \mathcal{X}_n$, there exists $\gamma\in (0,1]$ such that
\begin{align*}
    |u(Z_i)-u(Z_j)|\leq C(|Z_i-Z_j|^{\gamma}+\epsilon^{\gamma})
\end{align*}
where $C=C(\rho,\| u \|_{L^{\infty}(\mathcal{X}_n)})>0$.
\end{lemma}

To get the convergence near the boundary, we need the following lemma (for the proof, see \cite[Lemma 9.4]{ABP24}). 
 
\begin{lemma}\label{bd}
    Let $\Omega\subset \mathbb{R}^N$ be a bounded domain satisfying the uniform exterior ball condition and $\mathcal{X}_n\subset \Omega$ be a random data cloud. Suppose that the event \eqref{con11} and \eqref{con31} hold and that $u\in L^{\infty}(\mathcal{X}_n)$ satisfies
     \begin{align*}
\left\{ \begin{array}{ll}
   \mathcal{L}^+_{\mathcal{X}_{n},\epsilon} u\geq -\rho,\mathcal{L}^-_{\mathcal{X}_{n},\epsilon} u\leq \rho \,\,&\textrm{in} \,\, \mathcal{X}_{n}\setminus \mathcal{O}_{\epsilon},\\
    u=g\quad \quad \quad &\textrm{in}\,\,\mathcal{O}_{\epsilon}, \\
\end{array} \right.
\end{align*}
for some $\rho>0$, where $g\in C(\Gamma_{\epsilon_0})$. For every $\eta>0$ there exists $\delta>0$ such that 
\[
|u(Z_i)-g(x)|\leq \eta
\]
for every $Z_i\in \mathcal{X}_n$ and $x\in \partial \Omega$ with $|Z_i-x|<\delta$.
\end{lemma}
For any $x\in \partial \Omega$, by the uniform exterior ball condition, there exists $\xi\in \mathbb{R}^N\setminus \Omega$ such that $\overline{B}_r (\xi)\cap \overline{\Omega}=\{ x \}$. Looking back at the proof of Lemma \ref{bd} in \cite{ABP24}, the barrier functions $v,w\in L^{\infty}(\mathcal{X}_n)$ are given by:
\begin{align*}
    v(y)=g(x)+\frac{\eta}{2}+\theta \left[ (1+|x-\xi|)^{-\sigma}-(1+|y-\xi|)^{-\sigma} \right]= g(x)+\frac{\eta}{2}+\theta h_{\xi}(x,y),\\
    w(y)=g(x)-\frac{\eta}{2}-\theta \left[ (1+|x-\xi|)^{-\sigma}-(1+|y-\xi|)^{-\sigma} \right]= g(x)-\frac{\eta}{2}-\theta h_{\xi}(x,y),
\end{align*}
where $\theta>0$ and $\sigma=\sigma(\Omega)$ is a constant independent with $\epsilon$. Then for sufficiently large $\theta$,
\begin{align*}
    \mathcal{L}_{\mathcal{X}_n,\epsilon}^+ [v|_{\mathcal{X}_n}]&=\theta \mathcal{L}_{\mathcal{X}_n,\epsilon}^+ [h_{\xi}|_{\mathcal{X}_n}]=-\theta \frac{\sigma}{2}(R^2+1)^{-\sigma}\leq -\rho,\\
     \mathcal{L}_{\mathcal{X}_n,\epsilon}^+ [w|_{\mathcal{X}_n}]&=-\theta \mathcal{L}_{\mathcal{X}_n,\epsilon}^+ [h_{\xi}|_{\mathcal{X}_n}]=\theta \frac{\sigma}{2}(R^2+1)^{-\sigma}\geq \rho,
\end{align*}
where $R=R(\Omega,\xi)$ is a constant independent with $\epsilon$.
So we can find that $\delta$ and $\eta$ in Lemma \ref{bd} do not depend on $\epsilon$, but only depend on $p$, $N$, $\Omega$ and $\rho$. By \eqref{comparison} and Lemma \ref{bd}, we obtain asymptotic boundary continuity of $\{ u_k \}$. 

Now we state the convergence theorem using the transport map in \cite[Theorem 9.7]{ABP24}. This result is essential to derive our main result, which is the inference of Lemma \ref{aa} - \ref{ac}.
\begin{lemma}\label{conv}
    Let $\{ n_k \}$ be a sequence of positive integers and $\{ \epsilon_k \}$ be a sequence of positive real numbers such that 
   \[
   n_k\rightarrow \infty,\epsilon_k\rightarrow 0\quad \text{and}\quad \liminf_{k}(2n_k \text{exp}\{ -c_0n_k\epsilon_k^{3N+4+(N+2)a} \})<1.
    \]
For each $k\in \mathbb{N}$, let $\mathcal{X}^{(k)}_{n_k}\subset \Omega$ be a random data cloud in $\Omega$ and denote $\mathcal{O}_{\epsilon_k}=\mathcal{X}^{(k)}_{n_k}\cap \Gamma_{\epsilon_k}$. Let $u_k\in L^{\infty}(\mathcal{X}^{(k)}_{n_k})$ be such that 
 \begin{align*}
\left\{ \begin{array}{ll}
   \mathcal{L}^+_{\mathcal{X}_{n_k}^{(k)},\epsilon_k} u_k\geq -\rho,\mathcal{L}^-_{\mathcal{X}_{n_k}^{(k)},\epsilon_k} u_k\leq \rho \,\,&\textrm{in} \,\, \mathcal{X}^{(k)}_{n_k}\setminus \mathcal{O}_{\epsilon_k},\\
    u_k=g\quad \quad \quad &\textrm{in}\,\,\mathcal{O}_{\epsilon_k}, \\
\end{array} \right.
\end{align*}
for some positive $\rho>0$. Then with probability $1$, there exists $u:\Omega\rightarrow \mathbb{R}$ and a subsequence still denoted by $u_k$ such that 
\[
u_k\circ T_{\epsilon_k} \rightarrow u\quad \text{uniformly in}\,\,\overline{\Omega}
\]
as $k\rightarrow\infty$.
\end{lemma}

Finally, we check the convergence of $\{f_k\circ T_{\epsilon_k}\}$ and $\{u_k\circ T_{\epsilon_k}\}$. By Lemma \ref{aa}, we can find a function $f\in C(\overline{\Omega})$ satisfying 
\[
f_k\circ T_{\epsilon_k} \rightarrow f\textrm{ uniformly in }\overline{\Omega} ,\textrm{ as }k\rightarrow\infty
\]
(the subsequence are still denoted by $\{f_k\circ T_{\epsilon_k}\}$ for convenience).
Then by \eqref{comparison} we directly see that
\begin{align}\label{contr}
\mathcal{L}^+_{\mathcal{X}^{(k)}_{n_k},\epsilon_k} u_k\geq -C_f, \qquad  \mathcal{L}^-_{\mathcal{X}_{n_k,\epsilon_k}^{(k)}}u_k\leq C_f.
\end{align}
Due to lemma \ref{conv}, we can also get the convergence of $\{u_k\circ T_{\epsilon_k}\}$ (we still denote the subsequence by $\{u_k\circ T_{\epsilon_k}\}$). 

\subsection{Viscosity solution}
In the previous subsection, we established the convergence of the solutions to \eqref{gradpp}.
Now we demonstrate that the limiting function is a solution of the model problem in a viscosity sense.

We first recall the notion of a viscosity solution for 
\begin{align}\label{ope}
    \mathrm{div}(\phi^2|\nabla u|^{p-2}\nabla u)=f\text{  in  }\Omega. 
\end{align}
\begin{definition} Let $\Omega \subset \mathbb{R}^n$ and $u\in C(\Omega)$.
  We say $u$ is a viscosity subsolution (supersoslution, respectively) of \eqref{ope} if for every $x_0\in \Omega$ and $\xi\in C^2(\Omega)$ with $\nabla\xi(x_0)\neq 0$ such that $u-\xi$ has a local maximum (local minimum, respectively) at $x_0$ we have
    \[
    \mathrm{div}(\phi^2|\nabla u|^{p-2}\nabla u)\geq f. \qquad (\mathrm{div}(\phi^2|\nabla u|^{p-2}\nabla u)\leq f, \textrm{respectively}.)
    \]
 Moreover, we say that $u\in C(\Omega)$ is a viscosity solution of \eqref{ope} if $u$ is both a viscosity sub- and supersolution of \eqref{ope}.
\end{definition}

We now state the following result, which shows that the limiting function is a viscosity solution of our model problem.
\begin{theorem}\label{vis}
    The limit function $u$ is a viscosity solution to 
        \begin{align*}
\left\{ \begin{array}{ll}
  \frac{N+p}{N+2}\frac{\phi^{-2}}{|\nabla u|^{p-2}}\mathrm{div}(\phi^2|\nabla u|^{p-2}\nabla u)=f \,\,&\text{in} \,\,\Omega,\\
    u=g\quad \quad \quad \quad \qquad \quad&\text{in}\,\,\partial \Omega,\\
\end{array} \right.
\end{align*}
\end{theorem}
\begin{proof}

Let $\xi\in C^2$ be such that $u-\xi$ has a strict minimum at the point $x_0\in  \Omega$ with $u(x_0)=\xi(x_0)$. We want to show 
\[
 \frac{N+p}{N+2}\frac{\phi^{-2}}{|\nabla \xi|^{p-2}}\mathrm{div}(\phi^2|\nabla \xi|^{p-2}\nabla \xi)|_{x=x_0}\leq 0.
\]
\par By Lemma \ref{two} and \eqref{con11}, $\{ B_{\epsilon^{3+a}} (Z_i) \}_i$ is an open cover of the domain $\Omega$.
Thus, every point in $\Omega$ has one in the data cloud at a distance at most $\epsilon^{3+a}$. 
Since $\Tilde{u}_k=u_k\circ T_{\epsilon_k}$ uniformly converges to $u$ in $\overline{\Omega}$, we can find a sequence of points $x_k\in \mathcal{X}_{n_k}^{(k)}$ such that $x_k\rightarrow x_0$ as $k\rightarrow\infty$ and 
\begin{align}\label{yk}
\Tilde{u}_k(y)-\xi(y)\geq \Tilde{u}_k(x_k)-\xi(x_k)-\epsilon_k^3 \quad \textrm{for any} \ \  y\in\Omega.
\end{align}
\par Using Taylor's expansion to $v\in C^3(\Omega)$, we see that the difference 
\begin{align*}
  \bigg(\inf_{Z_j\in B_{\epsilon}^{\mathcal{X}_{n_k}^{(k)}}(Z_i)} v(Z_j)+\sup_{Z_k\in B_{\epsilon}^{\mathcal{X}_{n_k}^{(k)}}(Z_i)}v(Z_k)\bigg)-\left(\inf_{Z\in B_{\epsilon}(Z_i)} v(Z)+\sup_{Z'\in B_{\epsilon}(Z_i)} v(Z')   \right) 
    \end{align*}
is of order $O(\epsilon^3)$. 
 Since $\mathcal{L}_{\mathcal{X}_{n_k}^{(k)},\epsilon_k} {u}_k(x_k)=f
_k(x_k)$, we can observe the following (see (9.5) in \cite{ABP24}):
\begin{align*}
    \left| \frac{1}{\text{card}(B_{\epsilon}^{\mathcal{X}_n}(T_{\epsilon}(x)))}\sum_{Z_i\in B_{\epsilon}^{\mathcal{X}_n}(T_{\epsilon}(x))}  u(Z_i) - \fint_{B_{\epsilon}(x)} (u\circ T_{\epsilon})d\mu \right|\leq C\| u \|_{L^{\infty}(\mathcal{X}_n)}\epsilon^{2+a},
\end{align*}
and hence, $\xi$ satisfies the inequality
\begin{align*}
O(\epsilon_k) \geq \frac{\alpha}{2\epsilon_k^2}&\left(\inf_{y\in B_{\epsilon_k}(x_k)} \xi(y)+\sup_{z\in B_{\epsilon_k}(x_k)} \xi(z) -2\xi(x_k)  \right) \\&+\frac{\beta}{\epsilon_k^2}\left(\fint_{B_{\epsilon_k}(x_k)} \xi\circ T_{\epsilon_k}d\mu-\xi(x_k)\right),
\end{align*}
which gives us the desired result by taking the limit. Thus, $u$ is a viscosity supersolution. Similarly, we can prove $u$ is also a viscosity subsolution.

\par Finally, we verify the convergence near the boundary.
For any $y\in \partial \Omega$, we can find a random point $Z_y^{(k)}\in \mathcal{X}_{n_k,\epsilon_k}^{(k)}$ such that $\tilde{u}_k(y)=\tilde{u}_k(Z_y^{(k)})$, due to Lemma \ref{two}. Recalling \eqref{con11}, it holds $y\in B_{\epsilon^{3+a}}(Z_y^{(k)})$, so that 
\[
|Z_y^{(k)}-y|\leq 2\epsilon_k^3.
\]
Then we can check that for any $\eta>0$ and $y\in \partial \Omega$, there exists $N\in \mathbb{N}$ such that if $k>N$, then
\begin{align*}
    |\tilde{u}_k(y)-g(y)|<\eta.
\end{align*} 
 By Lemma \ref{bd}, take $k$ large enough such that $2\epsilon_k^3<\delta$, then it holds
\begin{align*}
    |\tilde{u}_k(y)-g(y)|=|\tilde{u}_k(Z_y^{(k)})-g(y)|
    \leq \eta.
\end{align*}
Then we get the convergence on the boundary and we can finish the proof.
\end{proof}
\begin{remark} 
There is another way to get this result, refer to Calder's proof in \cite[Theorem 1]{Cal}. 
Due to the uniform convergence, we can find a sequence of points $x_k\in \mathcal{X}^{(k)}_{n_k}$ such that $x_k\rightarrow x_0$ as $k\rightarrow\infty$ and ${u}_{k}-\xi$ attains its global minimum at $x_k$ i.e. ${u}_k(x)- \xi(x)  \geq u_k(x_k)-\xi(x_k) \quad \textrm{for any} \ \  x\in\mathcal{X}^{(k)}_{n_k}.$
Then we have $0=\mathcal{L}_{\mathcal{X}_{n_k}^{(k)},\epsilon_k} {u}_k(x_k)\geq \mathcal{L}_{\mathcal{X}_{n_k}^{(k)},\epsilon_k} \xi(x_k)$. 
By Lemma \ref{six}, we get
\[
\mathrm{div}(\phi^2|\nabla \xi|^{p-2}\nabla \xi)|_{x=x_0}\leq 0.
\]    
\end{remark}
We directly obtain the following corollary when $f=0$.
\begin{corollary}\label{fk}
Under the assumptions of Theorem \ref{mainthm}, let $u_k\in L^{\infty}(\mathcal{X}_{n_k}^{(k)})$ satisfy 
    \begin{align*}
\left\{ \begin{array}{ll}
   \mathcal{L}_{\mathcal{X}_{n_k}^{(k)},\epsilon_k} u_k=0 \,\,&\textrm{in} \,\, \mathcal{X}^{(k)}_{n_k}\setminus \mathcal{O}_{\epsilon_k},\\
    u_k=g\quad \quad \quad &\textrm{in}\,\,\mathcal{O}_{\epsilon_k}. \\
\end{array} \right.
\end{align*}
Then there exists 
a limit continuous function $u$ is the viscosity solution to 
    \begin{align*}
\left\{ \begin{array}{ll}
  \mathrm{div}(\phi^2|\nabla u|^{p-2}\nabla u)=0 \,\,&\textrm{in} \,\,\Omega,\\
    u=g\quad \quad \quad \quad \qquad \quad&\text{in}\,\,\partial \Omega, \\
\end{array} \right.
\end{align*}
with probability $1$.
\end{corollary}

\end{document}